\documentclass[10pt]{amsart}
\usepackage{fullpage, amssymb, amsmath}

\newtheorem{theorem}{Theorem}
\newtheorem{proposition}[theorem]{Proposition}
\newtheorem{lemma}[theorem]{Lemma}

\theoremstyle{definition}
\newtheorem{definition}[theorem]{Definition}
\newtheorem{question}[theorem]{Question}

\newtheorem{example}[theorem]{Example}

\theoremstyle{remark}
\newtheorem{remark}[theorem]{Remark}
\newtheorem{notation}[theorem]{Notation}

\numberwithin{equation}{section}
\numberwithin{theorem}{section}

\newcommand{\B}{\mathcal{B}}
\newcommand{\cc}{\mathfrak{c}}
\newcommand{\C}{\mathfrak{C}}
\newcommand{\CC}{\mathbb{C}}
\newcommand{\dens}{\text{dens}}
\newcommand{\eps}{\varepsilon}
\newcommand{\F}{\mathcal{F}}
\newcommand{\FF}{\mathbb{F}}

\newcommand{\M}{\mathcal{M}}
\newcommand{\MM}{\mathbb{M}}
\newcommand{\Mn}{\mathbb{M}_n}
\newcommand{\Mt}{\mathbb{M}_2}
\newcommand{\Mod}{\text{Mod}}
\newcommand{\N}{\mathcal{N}}
\newcommand{\NN}{\mathbb{N}}

\newcommand{\RR}{\mathbb{R}}
\newcommand{\R}{\mathcal{R}}
\newcommand{\s}{\mathcal{S}}
\newcommand{\T}{\text{II}_1}
\newcommand{\Th}{\text{Th}}
\newcommand{\tr}{\text{tr}}
\newcommand{\U}{\mathcal{U}}
\newcommand{\V}{\mathcal{V}}
\newcommand{\X}{\mathfrak{X}}

\newcommand{\ZZ}{\mathbb{Z}}
\begin{document}

\title{Model theory of operator algebras III: \\ Elementary equivalence and $\text{II}_1$ factors}

\author{Ilijas Farah}

\address{Department of Mathematics and Statistics\\
York University\\
4700 Keele Street\\
North York, Ontario\\ Canada, M3J
1P3\\
and Matematicki Institut, Kneza Mihaila 34, Belgrade, Serbia}

\email{ifarah@mathstat.yorku.ca}
\urladdr{http://www.math.yorku.ca/\symbol{126}ifarah}

\author{Bradd Hart}
\address{Department of Mathematics and Statistics\\
McMaster University\\ 1280 Main Street W.\\ Hamilton, Ontario\\
Canada L8S 4K1}

\email{hartb@mcmaster.ca}

\urladdr{http://www.math.mcmaster.ca/\symbol{126}bradd}
\thanks{The first two authors are partially supported by NSERC}

\author{David Sherman}

\address{Department of Mathematics\\
University of Virginia\\
P. O. Box 400137\\
Charlottesville, VA 22904-4137}
\email{dsherman@virginia.edu}

\urladdr{http://people.virginia.edu/\symbol{126}des5e}

\subjclass[2000]{46L10, 03C20, 03C98}
\keywords{continuous model theory, $\T$ factor, elementary equivalence, ultrapower, axiomatizability, pseudofiniteness}
\date{\today}

\begin{abstract}
We use continuous model theory to obtain several results concerning isomorphisms and embeddings between $\T$ factors and their ultrapowers.  Among other things, we show that for any $\T$ factor $\M$, there are continuum many nonisomorphic separable $\T$ factors that have an ultrapower isomorphic to an ultrapower of $\M$.  We also give a poor man's resolution of the Connes Embedding Problem: there exists a separable $\T$ factor such that all $\T$ factors embed into one of its ultrapowers.
\end{abstract}

\maketitle

\section{Introduction}

In this paper we study $\T$ factors using continuous model theory.  We develop ideas from \cite{FHS1,FHS2}, but the presentation here is intended to be self-contained and accessible to operator algebraists and other analysts who may not be familiar with logic.  Many of the results pertain to isomorphisms and embeddings between $\T$ factors and their ultrapowers.

Classical model theory is concerned with (first-order) logical properties, and some of its main techniques do not work well for analytic structures such as metric spaces.  \textit{Continuous model theory}, in which truth values are taken from bounded real intervals instead of the discrete set $\{T,F\}$, restores access to these logical techniques, in particular to fundamental theorems about ultraproducts.  Section \ref{S:back} of this paper offers some background on continuous model theory and the few basic tools used in this paper.  In the rest of this Introduction we outline our primary results, leaving full explanation of the terms to subsequent text.

Two objects that have the same logical properties are \textit{elementarily equivalent}.  By the continuous version of the Keisler-Shelah theorem, elementary equivalence is the same as the two objects having isomorphic ultrapowers.  This is a nontrivial relation, but much weaker than isomorphism itself: we show here that every $\T$ factor is elementarily equivalent to continuum many nonisomorphic separable $\T$ factors (Theorem \ref{T:main}).  We also go through several commonly-used operator algebraic properties and determine which are \textit{local} (=captured by first-order continuous model theory) or even \textit{axiomatizable}.

Another important logical relation is \textit{finite representability}, which induces a partial ordering based on the condition that one object embeds in an ultrapower of another.  This touches on issues around the Connes Embedding Problem, which asks whether every $\T$ factor embeds in an ultrapower of the hyperfinite $\T$ factor.  Although we do not offer any direct progress toward resolving this problem, our techniques easily imply that there are many ways of constructing a separable \textit{locally universal} $\T$ factor, i.e., one with the property that all $\T$ factors embed into one of its ultrapowers (Example \ref{T:ex}(2)).

A novel point of emphasis here is the identification of [structures modulo elementary equivalence] and [structures modulo mutual finite representability] with certain function spaces.  Some of the main logical notions are then expressible in terms of the topology and order in these function spaces.  We also introduce the class of \textit{pseudofinite} factors, those that are ``logic limits" of matrices.  It is known, but still surprising, that the hyperfinite $\T$ factor is not pseudofinite.  We give a short proof here (Theorem \ref{T:ps}) based on a neglected 1942 result of von Neumann.

The logical study of $\T$ factors is at a very early stage and may be unfamiliar to most of our targeted readership of analysts, so we present a fairly thorough exposition of the main ideas, including several results and questions from the literature (and answering a few of the latter).  On the other hand, we have consciously tried \textit{not} to include more logical machinery than necessary.  A sequel paper will collect results of a more general nature, relying on more substantial ideas from model theory.

\bigskip

\textbf{Acknowledgments.} We thank Junsheng Fang and Ward Henson for useful comments.

\section{Model theory, ultraproducts, and continuous model theory} \label{S:back}

\subsection{Model theory}
One starts with a first-order language for talking about a class of structures, called \textit{models} of the language.  (\textit{First-order} means that we allow quantification over a structure, but not the subsets of the structure -- that would be second-order.  We also do not distinguish typographically between a model and its underlying set, as is sometimes done.)  A basic example is the language of groups, with each group a model.  One expresses properties of a model by using the language to construct \textit{formulas}, which may or may not depend on variables.  We will write tuples of variables or elements with a bar, so that a typical formula is $\varphi(\bar{x})$.  The unquantified variables are called \textit{free}.  Now given a formula $\varphi(\bar{x})$ with $n$ free variables and a model $\M$, one can substitute an $n$-tuple $\bar{a} \in \M^n$ for $\bar{x}$ and see whether the resulting statement is true, writing $\M \vDash \varphi(\bar{a})$ if this is the case.  The symbol ``$\vDash$" denotes the relation of \textit{satisfaction}.  Formulas that do not have free variables are \textit{sentences}.  For instance, if we are working with the class of groups in the usual language, and we let $\varphi$ be the sentence $(\exists x) (x \neq e \wedge x^2 = e)$, then $\ZZ_6 \vDash \varphi$ but $\ZZ \not \vDash \varphi$.

The \textit{theory} of a model $\M$ is the set of sentences $\M$ satisfies, written $\Th(\M)$.  One also may consider the common theory of a class $K$ of models: $\Th(K) = \{\varphi \: | \: \M \vDash \varphi, \: \forall \M \in K\}$.  There is an obvious dual notion: for a set $\Sigma$ of sentences, $\Mod(\Sigma) = \{\M \: | \: \M \vDash \varphi, \: \forall \varphi \in \Sigma\}$.  A class of the form $\Mod(\Sigma)$ is said to be \textit{axiomatized} by $\Sigma$.

When two models cannot be distinguished by a first-order sentence, so that $\Th(\M) = \Th(\N)$, we say that $\M$ and $\N$ are \textit{elementarily equivalent} and write $\M \equiv \N$.  In general this is a weaker relation than isomorphism.  When $\M \subseteq \N$ and the truth value in $\M$ and the truth value in $\N$ are the same for any formula evaluated at a tuple from $\M$, we call $\M$ an \textit{elementary submodel} of $\N$.  This is stronger than saying $\M \subseteq \N$ and $\M \equiv \N$.

\subsection{Ultraproducts}
An \textit{ultrafilter} $\U$ on an index set $I$ can be defined in many ways.  We choose to view it as an element of the Stone-\v{C}ech compactification $\beta I$.  Thus ultrafilters on $I$ are either members of $I$ or limit points of $I$, the latter called \textit{free}.  The spaces $\beta I$ and $\beta I \setminus I$ are large (for $I$ infinite, their cardinality is $2^{2^{|I|}}$) and nonhomogeneous (their automorphism groups do not act transitively); many of their properties are sensitive to set-theoretic axioms beyond ZFC.

For a model $\M$ and an ultrafilter $\U$ on an index set $I$, the classical \textit{ultrapower} is
$$\M^\U = (\Pi_I \M)/\sim, \quad \text{where} \quad (x_j) \sim (y_j) \iff \U \in \overline{\{j \mid x_j = y_j\}}.$$
One should think of the equivalence relation as saying that the two sequences are ``equal at $\U$."  The ultrapower $\M^\U$ is a structure of the same kind as $\M$, and there is a canonical inclusion $\M \hookrightarrow \M^\U$ as (equivalence classes of) constant sequences.  This inclusion can be surjective, but typically not, and often there is a jump in cardinality.  For instance, when $\M$ is countably infinite and $\U$ is a free ultrafilter on $\NN$, then $|\M^\U| = \cc$, the cardinality of the continuum.  If one replaces $\Pi_I \M$ with $\displaystyle \Pi_{j \in I} \M_j$, the resulting object is called an \textit{ultraproduct} of the $\M_j$ and is written $\Pi_\U \M_j$.  Finally, any object $\N$ having an ultrapower isomorphic to $\M$ is called an \textit{ultraroot} of $\M$.

Ultrapowers were introduced by \L o\'{s} \cite{L} in 1955 and soon figured in several beautiful results in classical model theory.

\begin{theorem} \label{T:up}
Let $\M$ and $\N$ be models of the same language.
\begin{enumerate}
\item (\L o\'{s}'s theorem \cite{L}) For any ultrafilter $\U$, $\M$ is an elementary submodel of $\M^\U$, so in particular $\M^\U \equiv \M$.
\item (Keisler-Shelah theorem \cite{K,Sh1}) We have $\M \equiv \N$ if and only if there are ultrafilters $\U, \V$ such that $\M^\U \simeq \N^\V$.
\item (\cite{FMS}) A class of structures is axiomatizable if and only if it is closed under isomorphism, ultraproduct, and ultraroot.
\end{enumerate}
\end{theorem}

For a bounded metric space $(\M, d)$, one defines an ultrapower as follows:
\begin{equation} \label{E:up}
\M^\U = \Pi_I \M/\sim, \quad \text{where} \quad (x_j) \sim (y_j) \iff d(x_j, y_j) \overset{j \to \U}{\to} 0.
\end{equation}

If one applies \eqref{E:up} to the unit ball of a Banach space $\X$, one gets the unit ball of the Banach space
$$\X^\U = \ell^\infty_I(\X)/\{(x_j) \mid \lim_{j \to \U} \|x_j\| = 0\}.$$
This definition of $\X^\U$ seems to have first appeared in published work in 1972 \cite{DK}, but it is closely related to the slightly older concept of a nonstandard hull \cite{Lu}.  For $C^*$-algebras and other norm-based structures, the construction of the ultrapower is the same as for Banach spaces.

In this paper by a \textit{trace} $\tau$ on a von Neumann algebra $\M$ we will mean a weak*-continuous faithful tracial state.  A \textit{tracial von Neumann algebra} $(\M, \tau)$ will be a von Neumann algebra equipped with a given trace.  If $\M$ is a finite factor, then it has a unique trace that need not be specified.  The separable hyperfinite $\T$ factor will consistently be denoted $\R$.

On the norm unit ball of a tracial von Neumann algebra $(\M, \tau)$, the $L^2$-norm $\|x\|_2 = \sqrt{\tau(x^*x)}$ is a complete metric that implements the strong operator topology.  If one applies \eqref{E:up} to this metric space, one obtains the norm unit ball of the tracial von Neumann algebra
$$\M^\U = \ell^\infty_I(\M)/\{(x_j) \mid \lim_{j \to \U} \|x_j\|_2 = 0\}, \qquad \tau_{\M^\U}((x_j)) = \lim_{j \to \U} \tau(x_j).$$
These ultrapowers gained notice from their use in an important 1970 article of McDuff \cite{M}, but the underlying mathematics had been done in Sakai's 1962 notes \cite{Sa}, in fact mostly done in a 1954 paper of Wright \cite{W} about $AW^*$-algebras (predating \L o\'{s}!).

For analytic ultrapowers we have the same basic properties as in the classical case: the ultrapower is the same kind of structure as the original object, which it contains (usually properly) as constant sequences.  There is, however, a key difference with regard to the logical structure.  Loosely speaking, properties that only hold approximately in an object can hold precisely in its ultrapower.  The conversion from approximate to precise stems from saturation (defined in Section \ref{S:cmt}; Proposition \ref{T:gammarc}(3) is a specific instance of this conversion) and is perhaps the main reason that analytic ultrapowers are useful.  See \cite{S2009} for more on this theme.

But the strengthening from approximate to precise also entails that the model theoretic results of Theorem \ref{T:up} do not hold for analytic ultrapowers.  The following counterexample, which seems to be the simplest one to explain, is based on a similar discussion in \cite[p.28]{HI}.

\begin{example} \label{E:square}
Let $\ell^p_2$ denote the real 2-dimensional $L^p$ space.  Working in the language of real Banach spaces, consider the sentence
$$\varphi: (\exists x,y) \left[\|x\|=1 \wedge \|y\|=1 \wedge \left\|\frac{x+y}{2}\right\| = 1 \wedge \left\|\frac{x-y}{2}\right\| = 1\right].$$
The condition $\X \vDash \varphi$ says that $\X$ has a square on its unit sphere, or equivalently, $\ell^1_2 \overset{\sim}{\hookrightarrow} \X$.

If we take $\X = \oplus_{p=2}^\infty \ell^p_2$ ($\ell^2$-direct sum), then $\ell^1_2$ does not embed isometrically into $\X$, and $\X \not \vDash \varphi$.  Now let $\{x_j, y_j\}$ be the standard basis of the $\ell^j_2$ summand, so as $j$ increases these elements get closer and closer to satisfying the last two equations in $\varphi$.  Then for any $\U \in \beta \NN \setminus \NN$, the elements $(x_j), (y_j) \in \X^\U$ do satisfy the equations, and $\X^\U \vDash \varphi$.  This makes $\Th(\X) \neq \Th(\X^\U)$, which is enough to falsify all the statements in Theorem \ref{T:up}.
\end{example}

\subsection{Continuous model theory} \label{S:cmt}

One wants to rescue Theorem \ref{T:up} with a model theoretic framework that is appropriate for functional analysis, and over the years several have been proposed.
In our current research we employ (extensions of) the so-called \textit{model theory of metric structures}, of which basic expositions appeared only in 2008 \cite{BBHU} and 2010 \cite{BU}, although some of the main ideas go back to \cite{H} in 1976 and the more recent works \cite{HI,B}.  This is a ``continuous model theory," meaning that the truth value of a formula is taken from a bounded real interval (possibly depending on the formula) instead of the discrete set $\{T, F\}$.  The continuity of the range of the truth variable is of course not a new idea in logic \cite{CK} or even in model theory \cite{Kr}, but the syntax in this approach is particularly clear and flexible.

In the standard version, models must be equipped with a bounded metric $d$.  Formulas are built out of four ingredients.
\begin{itemize}
\item \textit{Terms} are meaningful expressions in the language, such as $x-y$ for Banach spaces.  This is no different from usual first-order logic.
\item The \textit{metric} is used to produce real numbers.  The truth of the classical formula $x=y$ is expressed by $d(x,y)$ having the value zero, so we might think of zero as generally corresponding to truth.
\item \textit{Continuous functions $\RR^n \to \RR$} can be applied.  These are the connectives.
\item We may take \textit{infima} and \textit{suprema} as variables run over the structure.  If a formula $\varphi(x)$ takes values in the range $[0,1]$, then the classical sentence $(\forall x) \varphi(x)$ is the same as the continuous sentence $\sup_x \varphi(x)$ having the value zero, although $(\exists x) \varphi(x)$ is stronger than $\inf_x \varphi(x)$ being zero.  Still we think of $\sup$ as $\forall$ and $\inf$ as $\exists$.
\end{itemize}
The value of a formula $\varphi(\bar{x})$ for a specific $\bar{a} \in \M^n$ is written $\varphi^\M(\bar{a})$.  We then define the \textit{theory} of a model $\M$ to be $\Th(\M) = \{\varphi \mid \varphi^\M = 0\}$.  The requirement ``$\varphi^\M = 0$" is called a \textit{condition} on $\M$, so that two models are elementarily equivalent when they satisfy the same conditions, and axiomatizable classes are those characterized by a set of conditions.  One defines ultrapowers as in \eqref{E:up}, with values of formulas governed by
\begin{equation} \label{E:limit}
\varphi^{\Pi_\U \M_j} ((\bar{a}_j)_j) = \lim_{j \to \U} \varphi^{\M_j} (\bar{a}_j).
\end{equation}

All concepts introduced so far have straightforward analogues in continuous model theory.  In fact \textit{all the statements in Theorem \ref{T:up} remain true} in continuous model theory.  (A comprehensive reference is \cite[Section 5]{BBHU}, although older papers such as \cite{St} effectively contain the same results.)  Note that a classical model can always be turned into a continuous model by equipping it with the discrete $\{0,1\}$-metric, so continuous model theory is a generalization of classical model theory.

For the sequel we will need two more model theoretic tools, this time only stating the continuous versions.  Again \cite{BBHU} is a good reference, although as stated Proposition \ref{T:sat}(2) requires techniques elaborated in the proof of \cite[Theorem 10.8]{HI}.  Recall that the \textit{density character} of a set in a topological space, denoted here $\dens(\cdot)$, is the minimal cardinality of a dense set.

\begin{theorem}[Downward L\"{o}wenheim-Skolem] \label{T:dls}
Let $X$ be a subset of a model $\M$ in a separable language.  Then there is a elementary submodel $\M_0 \subseteq \M$ such that $X \subseteq \M_0$ and $\dens(\M_0) \leq \dens(X) + \aleph_0$.  In particular, by taking $X = \varnothing$, it follows that every elementary equivalence class contains a separable model.
\end{theorem}

(The precise meaning of ``separable language" is somewhat technical (see \cite[Section 4.2]{FHS2}).  All languages considered in this paper are separable.)

Finally, for a cardinal $\kappa$, we say that a model $\M$ is \textit{$\kappa$-saturated} if the following holds.  Let $A$ be a subset of $\M$ with cardinality $< \kappa$.  Add to the language a constant symbol for each element of $A$; $\M$ is still a model in the augmented language.  Let $\{\varphi_i (\bar{x})\}$ be a set of formulas in the augmented language that is \textit{finitely approximately satisfied}, i.e., for any finite subset $\{\varphi_{i_k}\}$ and $\eps > 0$, there is a tuple $\bar{a} \in \M^n$ such that $\max_k |\varphi_{i_k}(\bar{a})| < \eps$.   Then there is already a tuple in $\M$ that satisfies all the formulas $\varphi_i(\bar{x})$.  Roughly speaking, $\kappa$-saturation says that elements that could exist in $\M$, subject to fewer than $\kappa$ parameters, actually do.

\begin{proposition} \label{T:sat}
Let $\M$ be a model.
\begin{enumerate}
\item For $\U \in \beta \NN \setminus \NN$, $\M^\U$ is $\aleph_1$-saturated.
\item For any cardinal $\kappa$, $\M$ has an ultrapower that is $\kappa$-saturated.
\end{enumerate}
\end{proposition}

In continuous model theory there are a language and a set of axioms whose models are exactly the unit balls of real Banach spaces \cite[Example 2.1(4)]{BBHU}.  (Recall that we want the metric to be bounded.)  Addition is not everywhere-defined, so one instead works with the averaging operation $(x,y) \mapsto \frac{x+y}{2}$.  The content of $\varphi$ from Example \ref{E:square} cannot be expressed in this language, but a substitute could be
$$\psi: \inf_{x,y} \left[ | \|x\|-1| + | \|y\|-1| + \left| \left\|\frac{x+y}{2}\right\|-1\right| + \left| \left\|\frac{x-y}{2}\right\|-1\right| \right].$$
The condition $\psi^\X = 0$ means that $\ell^1_2 \overset{1+\varepsilon}{\hookrightarrow} \X$, or equivalently, that $\X$ is not \textit{uniformly nonsquare} in the sense of James \cite{J}.  The Banach spaces that fail to be uniformly nonsquare thus form an axiomatizable class.

In \cite{FHS2} we gave explicit axiomatizations for $C^*$-algebras, tracial von Neumann algebras, and $\T$ factors.  The latter two present an extra difficulty because both the $L^2$ and norm metrics play roles.  Our solution was to develop a version of continuous model theory in which the underlying metric space of a model is a union of bounded sets $D_n$, intended as domains of quantification.  In the case of tracial von Neumann algebras one can take the domains to be balls of radius $n$ in the uniform norm, but equipped with the $L^2$-metric.  (In this framework one also regains access to operations like addition, understanding it as a collection of maps from $D_m \times D_n$ to $D_{m+n}$.)  We should note that different axiomatizations are mentioned in the two-page internet report \cite{BHJR}, which seems to represent the first effort to approach $\T$ factors from a model theoretic point of view.




\section{Axiomatizability and locality}

\subsection{A function representation for elementary equivalence classes of models}  Take any class of models that has been axiomatized in a language, and let $S$ be the set of sentences.  For a model $\M$, let $f_\M$ be the real-valued function on $S$ defined by $f_\M(\varphi) = \varphi^\M$.  Now $\Th(\M)$ is just the zero set of $f_\M$, but it is easy to see that it determines $f_\M$: the value of $f_\M(\varphi)$ is the real $c$ such that $\varphi - c$ belongs to $\Th(\M)$.  Thus $\M \equiv \N \iff f_\M = f_\N$.

\begin{notation}
We let $\C$ be the set of functions on $S$ of the form $f_\M$, subscripted if we want to identify the class of models, e.g. $\C_{\T}$.  As just mentioned, $\C$ can be identified with the set of elementary equivalence classes of models.  We endow $\C$ with the topology of pointwise convergence as functions on $S$.
\end{notation}

\begin{proposition}
$\C$ is compact, and axiomatizable subclasses are in 1-1 correspondence with the closed subsets of $\C$.
\end{proposition}

\begin{proof}
First recall that the range of any sentence lies in a bounded interval, so $\C$ lies in a product of bounded intervals, which is compact by Tychonoff.

An axiomatizable subclass $K$ is a collection of elementary equivalence classes, so corresponds to a subset of $\C$.  We claim that this subset must be closed.  For suppose $\{\M_j\} \subset K$ and $f_{\M_j} \to f$ pointwise.  Use \eqref{E:limit} and an appropriate ultrafilter $\U$ on the index set to compute, for any $\varphi \in S$,
$$f_{\Pi_\U \M_j}(\varphi) = \varphi^{\Pi_\U \M_j} = \lim_{j \to \U} \varphi^{\M_j} = \lim_{j \to \U} f_{\M_j}(\varphi) = f(\varphi).$$
Now $\Pi_\U \M_j \in K$ by Theorem \ref{T:up}(3), so $f = f_{\Pi_\U \M_j}$ must be in the corresponding subset of $\C$.

That any closed subset of $\C$ determines an axiomatizable class also follows from Theorem \ref{T:up}(3) and \eqref{E:limit}.
\end{proof}

\begin{remark}
We do not need to topologize the set $S$, although we could give it the weakest topology making all functions in $\C$ continuous.  Then separability of the language implies the separability of $S$.  As a consequence, the elementary equivalence class of a model $\M$ is determined by countably many values of $f_\M$, so that classification of $\T$ factors (also $C^*$-algebras, Banach spaces, etc.) up to elementary equivalence is \textit{smooth} in the sense of descriptive classification theory.
\end{remark}

We follow Henson \cite{H} in calling a subclass \textit{local} if it is closed under isomorphism, ultrapower, and ultraroot -- by Theorem \ref{T:up}(2) this is the same as being closed under elementary equivalence.  Thus local classes are in 1-1 correspondence with \textit{arbitrary} subsets of $\C$.  Note that the complement of an axiomatizable class need not be axiomatizable, but the complement of a local class is always local.  We freely identify a class with the property of belonging to that class.

Many Banach space properties are not even local, like ``$\X$ is reflexive" (not closed under ultrapower) or ``$\ell^1_2 \overset{\sim}{\hookrightarrow} \X$" (not closed under ultraroot, as we saw in Example \ref{E:square}).  From the preceding paragraphs one can use the topological structure of $\C$ to identify properties that are local but not axiomatizable, for instance the uniformly nonsquare Banach spaces.  (With $\psi$ as in Section 2.4, the set $\{f_\X \mid \: f_\X(\psi) > 0\}$ is open and not closed in $\C_{\text{Banach}}$.)  The capacity to define local properties syntactically is an advantage of continuous model theory over other approaches.

\subsection{Various classes of finite factors}

In the rest of this section we determine whether some commonly used properties of finite factors are local or axiomatizable.  For convenience fix $\U \in \beta \NN \setminus \NN$.

\subsubsection{Infinite dimensionality is axiomatizable; finite dimensionality is local but not axiomatizable} Infinite dimensionality is closed under ultraproducts and ultraroots.  Finite dimensionality is closed under ultraroots and ultrapowers, but not ultraproducts.

\subsubsection{Property $\Gamma$ is axiomatizable} \label{SS:gamma} A $\T$ factor $\M$ has \textit{property $\Gamma$} of Murray and von Neumann \cite{MvN} if for any finite set $F \subset \M$ and $\eps > 0$, there is a unitary $u \in \M$ such that $\tau(u) =0$ and $\max_{x \in F} \|[x,u]\|_2 < \eps$.

Consider the following sentences $\{\sigma_n\}$ about finite factors.  (Here and below, quantified variables range over $D_1$, the unit ball of the algebra.)
$$\sigma_n: \sup_{\bar{x}} \inf_y \|y^* y - I\|_2 + |\tau(y)| + \sum_{j=1}^n \|[x_j, y]\|_2.$$
The condition $\sigma_n^\M = 0$ says that for any $n$-tuple $\bar{x}$ in $\M$, there is an element $y$ that almost commutes with all members of $\bar{x}$, with $\|y^* y - I\|_2$ and $\tau(y)$ both small.  By a standard functional calculus argument, one can perturb $y$ slightly to an actual trace zero unitary \cite[Corollary 13.4.3]{SS}, at small expense to the quantities $\|[x_j, y]\|_2$.  The set $\{\sigma_n\}$ thus axiomatizes property $\Gamma$, cf. \cite[Definition 13.4.1]{SS}.

Let $m \geq 2$ and $L(\FF_m)$ be the $\T$ factor generated by the left regular representation of the free group $\FF_m$ on $\ell^2(\FF_m)$.  Murray and von Neumann showed that an element of $L(\FF_m)$ that nearly commutes with the unitaries associated to two of the generators must be close to the center, implying $\sigma_2^{L(\FF_m)} \neq 0$ \cite[Lemma 6.2.1-2]{MvN}.  Paired with the easy fact that $\R$ has $\Gamma$, this gives a culminating point of their work \cite[Theorem XVI$'$]{MvN}: $\R$ and $L(\FF_m)$ were the first example of nonisomorphic separable $\T$ factors.  We see here that they are not even elementarily equivalent.  It would be very interesting to obtain any specific information about the nonzero values of any $\sigma_n$.

\begin{remark}
In fact $\sigma_1^{L(\FF_m)} \neq 0$.  Let $u,v \in L(\FF_m)$ be unitaries corresponding to two of the group generators.  Then $W^*(u,v) = W^*(\text{Log}(u) + i \text{Log}(v)) \simeq L(\FF_2)$ is an irreducible subfactor of $L(\FF_m)$.  By \cite[Lemma 3.5]{FGL} the relative commutant $L(\FF_2)' \cap L(\FF_m)^\U$ is either nonatomic or $\CC$; the Murray-von Neumann calculation shows that it must be $\CC$, whence unitaries that nearly commute with $\text{Log}(u) + i \text{Log}(v)$ must be close to scalars and cannot have trace close to zero.
\end{remark}

\begin{question} \label{T:neggamma}
Are the non-$\Gamma$ $\T$ factors an axiomatizable class?  They would not be axiomatizable if and only if there are non-$\Gamma$ $\{\M_j\}$ such that $\sigma_n^{\M_j} \to 0$ for every $n$.  Since $\sigma_m^\M \leq \sigma_n^\M$ for $m \leq n$, this is the same as saying that the nonzero values of $\sigma_n$ accumulate at zero for every $n$.
\end{question}

It is immediate that a $\T$ factor $\M$ has $\Gamma$ if and only if
\begin{equation} \label{E:gamma}
\text{the relative commutant in $\M^\U$ of any finite set in $\M$ contains a trace zero unitary.}
\end{equation}
(Given a finite set in $\M$, obtaining unitaries by letting $\eps \to 0$ in the $\Gamma$ condition is essentially the same as having a representing sequence for a trace zero unitary in the relative commutant.)

\begin{proposition} \label{T:gamma1}
For a $\T$ factor $\M$, not necessarily separable, one still obtains a characterization of $\Gamma$ if any subset of the following changes to \eqref{E:gamma} are made: ``finite" to ``countable", ``$\M$" to ``$\M^\U$", ``contains a trace zero unitary" to ``is nontrivial".
\end{proposition}
Some of the equivalences between the eight conditions in Proposition \ref{T:gamma1} were mentioned without proof in \cite[Example 4.2]{S2009}, and we feel compelled to justify them in this paper.  But the details are largely unrelated to our main narrative, so we postpone them to an appendix.  For \textit{separable} factors Proposition \ref{T:gamma1} is well-known, as is the equivalence between $\Gamma$ and the condition $\M' \cap \M^\U \neq \CC$.  One might ask whether this equivalence persists for nonseparable factors -- we show next that it does not, and in fact the condition $\M' \cap \M^\U \neq \CC$ is not even local.  The reason is that, under our standing assumption that $\U \in \beta \NN \setminus \NN$, $\M^\U$ is only guaranteed to be $\aleph_1$-saturated (Proposition \ref{T:sat}(1)).  But property $\Gamma$ is equivalent to having nontrivial relative commutant in \textit{some} ultrapower.

\begin{proposition} \label{T:gammarc} ${}$
\begin{enumerate}
\item There exists a (nonseparable) $\T$ factor $\M$ such that $\M' \cap \M^\U = \CC$ for $\U \in \beta \NN \setminus \NN$, but the relative commutant in $\M$ of any countable set in $\M$ contains $\Mt$ unitally (so $\M$ has $\Gamma$).
\item The condition $\M' \cap \M^\U \neq \CC$ is not local.
\item Property $\Gamma$ is equivalent to the existence of an ultrafilter $\V$ (on some possibly uncountable set) such that $\M' \cap \M^\V \neq \CC$.
\end{enumerate}
\end{proposition}

\begin{proof}
(1): Our example will be the union of an increasing family of separable $\T$ factors $\{\M_\alpha | \: \alpha < \aleph_1\}$ defined by transfinite recursion.  Let $\M_1 = \R$.  Let $\N_1$ be the tracial free product $\M_1 * \Mt$, which is a separable non-$\Gamma$ $\T$ factor by \cite[Theorem 11]{Ba}.  Now we assume that $\M_\alpha$ and $\N_\alpha$ have been defined for all $\alpha < \beta$, and we explain how to construct $\M_\beta$ and $\N_\beta$.  If $\beta$ is a successor ordinal, say $\beta = \gamma + 1$, set $\M_\beta = \Mt \otimes \N_\gamma$.  If $\beta$ is a limit ordinal, set $\M_\beta = \overline{\cup_{\gamma < \beta} \M_\gamma}$.  (The closure is the usual one.  Represent the union on the Hilbert space coming from the GNS construction for the coherent trace; the closure in the strong operator topology is again a $\T$ factor.)  In either case, set $\N_\beta = \M_\beta * \Mt$.  For ordinals $\alpha < \beta$ we have $\M_\alpha \subset \N_\alpha \subset \M_\beta$, and we set $\M = \cup_{\alpha < \aleph_1} \M_\alpha$.  As explained in \cite[proof of Theorem 2.5]{S2010}, this last union is a $\T$ factor (no closure required).

Let $\{x_j\} \subset \M$ be countable.  For each $j$ there is some $\alpha_j < \aleph_1$ such that $x_j \in \M_{\alpha_j}$.  Then $\sup \alpha_j < \aleph_1$, so that
$$\{x_j\} \subset \M_{\sup \alpha_j} \subset \N_{\sup \alpha_j} = \N_{\sup \alpha_j} \otimes I \overset{(*)}{\subset} \N_{\sup \alpha_j} \otimes \Mt = \M_{(\sup \alpha_j) + 1} \subset \M.$$
From (*), $\{x_j\}' \cap \M$ contains a copy of $\Mt$.

Finally we prove that central $\U$-sequences are trivial.  Suppose $(x_j) \in \M' \cap \M^\U$.  As in the previous paragraph, all $x_j$ belong to some $\N_\alpha$, so $(x_j) \in \N_\alpha' \cap \N_\alpha^\U$, but the latter is $\CC$ since $\N_\alpha$ lacks $\Gamma$.

(2): Let $\M$ be as in (1) and use the downward L\"{o}wenheim-Skolem theorem to find a separable elementary submodel $\N \subseteq \M$.  Since $\M$ has $\Gamma$, so does $\N$, and by separability $\N' \cap \N^\U \neq \CC$.

(3): Suppose $(y_i) \in \M' \cap \M^\V$ and $\text{dist}((y_i), \CC) = c > 0$.  Then for any finite $\{x_k\} \subset \M$ and $j \in \NN$, there is $i_j$ such that $\max_k \{\|[y_{i_j}, x_k]\|\} < \frac{1}{j}$ and $\text{dist}(y_{i_j}, \CC) > \frac{c}{2}$.  This implies that the relative commutant in $\M^\U$ of $\{x_k\}$ contains the nonscalar element $(y_{i_j})$, so $\M$ has $\Gamma$ by Proposition \ref{T:gamma1}.

On the other hand, suppose $\M$ has $\Gamma$, and let $\{x_i\}_{i \in I}$ be a dense set in $\M$.  Using Proposition \ref{T:sat}(2), there exists an ultrapower $\M^\V$ that is $\kappa$-saturated for some $\kappa$ greater than $|I|$.  Consider the collection of conditions $\{\|[u, x_i]\| = 0\}_i \cup \{\|u^*u - I\| = 0\} \cup \{\tau(u)=0\}$.  By $\Gamma$ any finite subset has a common solution in $\M^\V$, so by saturation the entire collection has a common solution, i.e., a trace zero unitary in $\M' \cap \M^\V$.
\end{proof}

\begin{remark}
The Effros-Mar\'{e}chal topology on von Neumann subalgebras of $\B(\ell^2)$ (see \cite{HW1,HW2}) can be relativized to the Borel subset of $\T$ subfactors of $\B(\ell^2)$.  From this one can induce a quotient topology on elementary equivalence classes of $\T$ factors, or equivalently on $\C_{\T}$, and we point out that it is distinct from the pointwise topology on $\C_{\T}$.  For instance, factors with $\Gamma$ are pointwise closed, being axiomatizable.  But in the quotient of the Effros-Mar\'{e}chal topology, the closure of the factors with $\Gamma$ contains the closure of the hyperfinite factors, which is exactly the set of factors that embed in an ultrapower of $\R$ \cite[Theorem 5.8]{HW2}.  It is known that the latter contains many non-$\Gamma$ factors, such as the free group factors (see \cite{O2}).
\end{remark}


\subsubsection{Axiomatizability of McDuffness depends on the definition} \label{SS:mcd}
A separable $\T$ factor that satisfies any one of the equivalent conditions in Proposition \ref{T:mcd} below is called \textit{McDuff}.  Depending on the choice of definition for general $\T$ factors, McDuffness may or may not be axiomatizable -- the same lesson as in Proposition \ref{T:gammarc}(2).

\begin{proposition} \label{T:mcd}
For a separable $\T$ factor $\M$, the following conditions are equivalent:
\begin{enumerate}
\item for any finite subset $\{x_j\} \subset \M$, $W^*(\{x_j\})' \cap \M^\U \supseteq \Mt$;
\item same as (1), making one or both of the following changes: ``finite" to ``countable", ``$\{x_j\} \subset \M$" to ``$\{x_j\} \subset \M^\U$";
\item $\M' \cap \M^\V \supseteq \Mt$ for some ultrafilter $\V$ (possibly on an uncountable set);
\item $\M' \cap \M^\U \supseteq \Mt$;
\item $\M' \cap \M^\U$ is noncommutative;
\item $\M \simeq \M \otimes \R$.
\end{enumerate}
For arbitrary $\T$ factors, properties (1)-(3) define the same axiomatizable class, while properties (4)-(6) are not even local.
\end{proposition}

\begin{proof}[Idea of proof]
The original paper of McDuff \cite{M} covers the equivalence of (4)-(6) in the separable case.

Note that property (1) can be axiomatized by $\{\sigma'_n\}$, where
$$\sigma'_n: \sup_{\bar{x}} \inf_y \|yy^*y - y\|_2 + \|y^*y + yy^* - I\|_2 + \sum_{j=1}^n \|[x_j, y]\|_2,$$
cf. \cite[Remark 2.12]{Sdiv}.  The idea is that if $y$ makes this quantity nearly zero, it almost commutes with the $\{x_j\}$ and is very close to a partial isometry between equivalent complementary projections; at small cost to the commutator norms, $y$ can be perturbed to a matrix unit $e_{12}$ in some copy of $\Mt$.  Again it would be significant to calculate any nonzero values for these sentences.

Arguments similar to Proposition \ref{T:gamma1} and Proposition \ref{T:gammarc}(3) then establish that (1)-(3) are equivalent for arbitrary factors.  Back in the separable case, obviously [(2) with ``countable"] $\Rightarrow$ (4) $\Rightarrow$ (3), so all six are equivalent.

The example in Proposition \ref{T:gammarc}(1,2) shows that none of (4)-(6) is local.
\end{proof}

\subsubsection{Hyperfiniteness is not local}  This follows immediately from Theorem \ref{T:main} below, as any separable $\M \equiv \R$ with $\M \not \simeq \R$ is necessarily not hyperfinite.

\subsubsection{Primeness is not local} A $\T$ factor is said to be \textit{prime} if it cannot be written as the tensor product of two $\T$ factors.  Obviously not all separable $\T$ factors are prime -- consider any tensor product -- but $\M^\U$ is prime for every separable $\M$ \cite[Theorem 4.5]{FGL}.

\subsubsection{None of the following classes are local, unless the last is just $\CC$: finite factors containing a generator, finite factors containing an irreducible element, finite factors containing an element with no nontrivial invariant projection}

There exist singly-generated $\T$ factors, such as $\R$, but $\R^\U$ is nonseparable and so cannot even be countably-generated.  A similar principle (formally, the upward L\"{o}wenheim-Skolem theorem) shows that no bound on the number of generators is local.

An element $x \in \M$ is said to be \textit{irreducible} if $W^*(x)' \cap \M = \CC$.  Every separable $\T$ factor has an irreducible element, since it has an irreducible hyperfinite (so singly-generated) subfactor \cite[Corollary 4.1]{P1981}.  But if $\M$ is a separable $\T$ factor satisfying $\sigma_1^\M = 0$ (in particular, if $\M$ has $\Gamma$), then $\M^\U$ does not have an irreducible element.  For given any $(x_j) \in \M^\U$, choose $\{u_j\}$ to be a sequence of trace zero unitaries with $\|[u_j, x_j]\|_2 < \frac1j$; then $(u_j)$ is a nonscalar element of $ W^*((x_j)_j)' \cap \M^\U$.

If $x,p \in \M$ satisfy $xp = pxp$, and $p$ is a projection, then in any representation of $\M$ the range of $p$ is an invariant subspace of $x$.  We call $p$ an \textit{invariant projection} for $x$.  For $x \in \Mn$, the projection onto the span of an eigenvector is invariant.  But it is unknown whether every element of a $\T$ factor has a nontrivial invariant projection.  Fang and Hadwin \cite[Theorem 2.1]{FH} showed that for $\M$ a separable $\T$ factor, every element of $\M^\U$ has many nontrivial invariant projections, so this class cannot be local unless it is just $\CC$.

\section{Isomorphic and nonisomorphic ultrapowers}

\textit{In this section all ultrapowers are based on free ultrafilters of $\NN$, except where noted.}

Given two separable $\T$ factors $\M$ and $\N$, one might ask whether their ultrapowers are isomorphic.  Ultrapowers of infinite dimensional objects are nonseparable, making it impractical to construct an isomorphism -- in fact, as we recall below, the answer can depend on the choice of ultrafilters.  A better question is whether the two factors have any isomorphic ultrapowers at all (not necessarily based on free ultrafilters of $\NN$); by Theorem \ref{T:up}(2), this is the same as asking whether they are elementarily equivalent.  In this section we present some questions and answers on this topic.

\textbf{Q.}  If $\M$ is a separable model, must all ultrapowers of $\M$ be isomorphic?

\textbf{A.} This is the main result of \cite{FHS2}.  To summarize it, we first recall that a theory $T$ is said to have the \textit{order property} if there exists a formula $\varphi(\bar{x}, \bar{y})$ and a model of $T$ containing an infinite sequence of tuples $(\bar{a}_j)$ such that $\varphi(\bar{a}_i, \bar{a}_j)$ is 0 if $i \leq j$ and 1 if $i > j$.  In other words, one can encode $(\NN, \leq)$ in a model of $T$.

\begin{theorem} \label{T:ch} $($\cite{FHS1,FHS2}$)$
Let $\M$ be separable in a separable language.
\begin{enumerate}
\item All ultrapowers of $\M$ are isomorphic if and only if at least one of the following conditions is met:
\begin{enumerate}
\item the continuum hypothesis (CH) is assumed;
\item $\Th(\M)$ does not have the order property.
\end{enumerate}
\item If $\M$ is a $C^*$-algebra, then $\Th(\M)$ has the order property unless $\M$ is finite-dimensional.  If $\M$ is a tracial von Neumann algebra, then $\Th(\M)$ has the order property unless $\M$ is of type I.
\end{enumerate}
So for separable $\T$ factors and separable infinite-dimensional $C^*$-algebras, the uniqueness of the ultrapower is equivalent to CH.
\end{theorem}

\begin{remark} \label{T:rem} ${}$
\begin{enumerate}
\item Ge and Hadwin \cite{GH} had shown the sufficiency of CH in part (1) (although they specialized to $\T$ factors).  Actually the general result is quite easy: all the ultrapowers are $\aleph_1$-saturated and have density character $\cc$; if $\cc = \aleph_1$ the conclusion follows from the fact that two $\kappa$-saturated elementarily equivalent models of density character $\kappa$ are isomorphic \cite[Proposition 4.13]{FHS2}.
\item The article \cite{FHS1} also proves a version of (2) for relative commutants $\M' \cap \M^\U$, as follows.  The relative commutant of a separable infinite-dimensional $C^*$-algebra is unique up to isomorphism if and only if CH is assumed, answering a question of Kirchberg to the first named author and generalizing the main result in \cite{F}.  For a separable $\T$ factor $\M$, the isomorphism type of the relative commutant depends on the ultrafilter if and only if CH is denied and $\M$ is McDuff.  This completes the answer (started in \cite{GH}) to a question from McDuff's original paper \cite[Question (i), p.460]{M}.
\item In \cite{FHS2} it is shown that a theory $T$ has the order property if and only if it is not \textit{stable} (see \cite[Section 5]{FHS2} for the definition). 
      It seems that the relation between stability and nonisomorphic ultrapowers is not well known even in classical model theory.
\item Shelah and the first-named author \cite{FS} have sharpened Theorem \ref{T:ch} by showing that the number of nonisomorphic ultrapowers is either 1 or $2^\cc$.
\end{enumerate}
\end{remark}

\textbf{Q.} If two separable $\T$ factors are elementarily equivalent, do they have isomorphic ultrapowers based on free ultrafilters of $\NN$?

\textbf{A.} If CH is assumed, the answer is yes, and any free ultrafilters will do.  This is based on the same fact cited in Remark \ref{T:rem}(1).  Without CH we do not know the answer, but for structures other than $\T$ factors it is known that one may need ultrafilters on uncountable index sets \cite{Sh2}.

\textbf{Q.} Can two nonisomorphic separable $\T$ factors fail to have isomorphic ultrapowers?

\textbf{A.} The first explicit answer we can find is in \cite{FGL}, where it was noted that property $\Gamma$ passes to ultrapowers and ultraroots, and also that a non-McDuff factor cannot have an ultrapower isomorphic to an ultrapower of $\R$.  These are instances of \L o\'{s}'s theorem and were covered in Sections \ref{SS:gamma} and \ref{SS:mcd}.

\textbf{Q.} Can two nonisomorphic separable $\T$ factors have isomorphic ultrapowers?

\textbf{A.} This does not seem to have been explicitly displayed in the literature, but must be known to the experts.  We first need to recall the definition of the \textit{fundamental group} $\F(\M)$ for a $\T$ factor $\M$.  With $\tau$ the unique tracial state of $\M$ and $\tr$ the usual trace on $\B(\ell^2)$, $\tr \otimes \tau$ is a tracial weight on the $\text{II}_\infty$ factor $\B(\ell^2) \bar{\otimes} \M$.  For $\lambda \in \R_+$, we set $\M^\lambda = p(\B(\ell^2) \bar{\otimes} \M)p$, where $p$ is any projection in $\B(\ell^2) \bar{\otimes} \M$ such that $(\tr \otimes \tau)(p) = \lambda$; $\M^\lambda$ is well-defined up to isomorphism.  Finally $\F(\M)$ is the subgroup of $\RR_+$ defined by $\{\lambda \mid \: \M \simeq \M^\lambda\}$.

The important fact here is that $\F(\M)$ need not be closed in $\RR_+$ (in which case it is necessarily dense).  This was first established in 1983 by Golodets-Nessonov \cite{GN1}; see \cite{GN2} for the English version or \cite{P1995} for a different proof.  Given such $\M$, let $1 > \lambda \notin \F(\M)$, so that $\M \not \simeq \M^\lambda$.  We claim that $\M^\U \simeq (\M^\lambda)^\U$.  For let $\{\lambda_n\} \subset\F(\M) \cap (0,1)$ be an increasing sequence converging to $\lambda$, and let $\theta_n$ be an isomorphism from $\M$ to $p_n(\M^\lambda)p_n \simeq \M^{\lambda_n}$, where $p_n$ is a projection in $\M^\lambda$ having trace $\lambda_n/\lambda$.  It is easy to see that the map $\Pi \theta_n: \ell^\infty(\M) \to \ell^\infty(\M^\lambda)$ descends to an isomorphism from $\M^\U$ to $(\M^\lambda)^\U$.  (It is surjective because its range is the ultraproduct $\Pi_\U (p_n \M^\lambda p_n)$, which is strongly dense in $(\M^\lambda)^\U$ and thus equal to it.)  Thus the factors $\{\M^\lambda \mid \lambda \in \RR_+\}$ all have isomorphic ultrapowers, although they are not all isomorphic to each other.

The next answer supersedes this one.

\textbf{Q.} Are there any separable $\T$ factors $\N$ such that for separable $\M$, $\M^\U \simeq \N^\U$ implies $\M \simeq \N$?  In \cite[end of Section 4]{FH}, it was asked whether $\R$ has this property.

\textbf{A.} The following new theorem gives a strong negative answer.

\begin{theorem} \label{T:main}
Let $\N$ be any $\text{II}_1$ factor.  There are $\cc$ separable nonisomorphic $\text{II}_1$ factors that are elementarily equivalent to $\N$.  If CH is assumed and $\N$ is separable, then for any of these factors $\M$ we have $\M^\U \simeq \N^\U$.
\end{theorem}

\begin{proof}
The article \cite{NPS} displays an uncountable family of separable $\text{II}_1$ factors $\{\M_\alpha\}$ such that each $\M_\alpha$ embeds in $\R^\U$, and at most countably many of the $\M_\alpha$ embed in any given $\text{II}_1$ factor.  Given $\N$, we always have $\R \hookrightarrow \N$, and we can take the ultrapower of this inclusion to conclude that for each $\alpha$, $\M_\alpha \hookrightarrow \R^\U \hookrightarrow \N^\U$.  Use the downward L\"{o}wenheim-Skolem theorem (Theorem \ref{T:dls}) to find a separable elementary submodel $\N_\alpha$ of $\N^\U$ containing $\M_\alpha$.  Then note that each of the $\N_\alpha$ can be isomorphic to at most countably many other $\N_\beta$, as such an isomorphism implies that $\M_\beta$ embeds in $\N_\alpha$.  So there are uncountably many nonisomorphic $\N_\alpha$, each elementarily equivalent to $\N$.  The second sentence of the theorem follows from the argument in Remark \ref{T:rem}(1).
\end{proof}

\section{Pseudofiniteness} \label{S:fmp}

In classical model theory, an infinite model $\M$ is \textit{pseudofinite} if it is elementarily equivalent to an ultraproduct of finite models.  In continuous model theory one should replace finiteness with compactness (see \cite{GL}); in particular, when all models are balls in vector spaces, we will say that an infinite-dimensional model $\M$ is \textit{pseudofinite} if it is elementarily equivalent to an ultraproduct of finite-dimensional models.  By definition pseudofiniteness is axiomatizable.

Thus a model $\M$ is pseudofinite if there are finite-dimensional $\M_\alpha$ with $f_{\M_\alpha} \to f_\M$ in $\C$.  One might think of $\M$ as a limit ``in logic" of finite-dimensional structures.

Pseudofiniteness of (real) Banach spaces was studied in several papers of Henson and Moore about thirty years ago.  They turned up several interesting phenomena, but really only scratched the surface; it is still unknown whether $L^p$ ($p \neq 2$) is pseudofinite.  See \cite{Mo} for some overview of the main results.  Even for $C_0(K)$ spaces subtle things happen: $c_0$ is not pseudofinite, while for compact $K$ pseudofiniteness of $C(K)$ is equivalent to $K$ being totally disconnected and having a dense set of isolated points.

Turning now to finite factors, we see by Theorem \ref{T:main} that there are continuum many nonisomorphic separable pseudofinite $\T$ factors.  These are ``logic" limits of matrices, so an operator algebraist would guess that $\R$ is pseudofinite, but as we show below, it is not.  In 2007 the third named author realized that this follows directly from a neglected 1942 paper of von Neumann \cite{vN}.  At that time some experts already knew the result, as a consequence of recent techniques.  In a case of synchronicity, it was first written up in 2008 by Fang and Hadwin \cite{FH}, whose proof follows the same general lines as ours, more or less substituting results of Szarek \cite{Sz} for those of von Neumann. 

\begin{theorem} \label{T:ps} $(\sim$\cite[Section 4]{FH}$)$
Pseudofinite factors do not have $\Gamma$ (and thus $\R$ is not pseudofinite).
\end{theorem}

\begin{proof}
The main result of von Neumann's paper \cite{vN} is Theorem 23.1, which says that for every $\delta > 0$ there is $\eps > 0$ and a sequence of matrices $x_n \in \mathbb{M}_n$ such that
\begin{equation} \label{E:vN}
y \in \mathbb{M}_n, \quad \|[y,y^*]\|_2, \|[x_n,y]\|_2 < \eps \qquad \Rightarrow \qquad \|y - \tau(y)I\|_2 < \delta.
\end{equation}

For specificity take $\eps$ and $\{x_n\}$ corresponding to $\delta = \frac14$ in von Neumann's theorem; we may assume $\eps < \frac34$.  (Actually a simple calculation in $2 \times 2$ matrices implies that $\eps$ cannot be $\geq 3/4$.)  Then
\begin{equation} \label{E:vN2}
\sigma_1^{\Mn} = \sup_{  x \in {   {(\Mn)}_{\leq 1}  }   } \inf_{y \in {{(\Mn)}_{\leq 1}}} \|y^* y - I\|_2 + |\tau(y)| + \|[x, y]\|_2 \geq \inf_{y \in {{(\Mn)}_{\leq 1}}} \|y^* y - I\|_2 + |\tau(y)| + \|[x_n, y]\|_2.
\end{equation}
If $\sigma_1^{\Mn} < \frac{\eps}{2}$, then there is a specific contraction $y \in (\Mn)_{\leq 1}$ such that each of the three summands on the right hand side of \eqref{E:vN2} is $< \frac{\eps}{2}$.  We show that this leads to a contradiction, as follows.

Since $y^*y - I$ and $yy^* - I$ are unitarily conjugate, they have the same $L^2$-norm.  Then
$$\|y^*y - I\|_2 < \eps/2 \qquad \Rightarrow \qquad  \|[y,y^*]\|_2 = \|(yy^* - I) + (I - y^*y)\|_2 \leq \|yy^* - I\|_2 + \|I - y^*y\|_2 < (\eps/2) + (\eps/2) = \eps.$$
Now both conditions in \eqref{E:vN} are met, so $\|y - \tau(y)I\|_2 < \frac14$.  Using the elementary inequality $\|ab\|_2 \leq \|a\|_\infty \|b\|_2$ ($\|\cdot\|_\infty$ denotes the operator norm), compute
$$\eps/2 > \|y^*y - I\|_2 \geq \|I\|_2 - \|y^*y\|_2 \geq 1 - \|y^*\|_\infty \|y\|_2 \geq 1 - 1 \cdot \|\tau(y)I + (y - \tau(y)I)\|_2 \geq 1 - [(\eps/2) +(1/4)].$$
But this contradicts $\eps < 3/4$.

Therefore $\sigma_1^{\Mn} \geq \frac{\eps}{2}$ for all $n$.  If $\M$ is a pseudofinite factor, then it is elementarily equivalent to an ultraproduct $\Pi_\U \Mn$, and $\sigma_1^\M = \lim_{n \to \U} \sigma_1^{\Mn} \geq \frac{\eps}{2}$.
\end{proof}

Separable pseudofinite factors are in some sense the simplest, being the closest to matrix algebras.  Yet $\R$ is not pseudofinite, and in fact no specific pseudofinite factors have been identified!

\begin{question}
Are all pseudofinite factors elementarily equivalent?  This is just another way of asking whether $f_{\Mn}$ converges as $n \to \infty$.  It may have first appeared in print in \cite{Li}, but it is certainly older than that and was also posed to the first named author by Popa.  Assuming CH, it is the same as asking whether the ultraproducts $\Pi_\U \Mn$ are all isomorphic as $\U$ runs over $\beta \NN \setminus \NN$.  (This follows from Theorem \ref{T:ch}.  If CH fails, they are not all isomorphic.)
\end{question}

\section{Finite representability, local universality, and the Connes Embedding Problem} \label{S:fr}

One of the most actively studied issues in $\T$ factors today is the \textit{Connes Embedding Problem (CEP)}, which in its original form \cite{C} asks whether every separable $\T$ factor embeds in $\R^\U$, for $\U \in \beta \NN \setminus \NN$.  See \cite{O2} for a survey of its equivalents in operator algebras, and \cite{Pe} for an exposition of its relation to hyperlinear and sofic groups.  In this section we interpret CEP and related issues in model theoretic terms.  The poset framework we discuss is old news in Banach space theory \cite{Sch}, and we do not claim that these formulations represent any true progress toward solving CEP.  Nonetheless they unify concepts from different classes of models, and they may suggest new problems.

\subsection{Another function representation} Return to the general setup of a class axiomatized in continuous model theory.  Let $S' \subset S$ be the subset of $\sup$-sentences, i.e., those of the form $\sup_{\bar{x}} \varphi(\bar{x})$ for some quantifier-free formula $\varphi$.  We write $f'_\M$ for the restriction $f_\M |_{S'}$.  Once again $f'_\M$ is determined by its zero set, which one might term the \textit{universal theory} of $\M$.  As an extension of a theorem in classical model theory, we have
\begin{equation} \label{E:embed}
f'_\M \leq f'_\N \quad \iff \quad \M \text{ embeds in an ultrapower of }\N.
\end{equation}
(See \cite[Proposition 13.1]{HI}.)  The condition on the left means, roughly, that any finite configuration in $\M$ can also be found (approximately) in $\N$.  One says that $\M$ is \textit{finitely representable} in $\N$.  Unlike the situation for elementary equivalence, CH is not needed to conclude that if $\M$ and $\N$ are separable, then $\M$ embeds in some ultrapower of $\N$ based on a free ultrafilter of $\NN$ if and only if it embeds in all ultrapowers of $\N$ based on free ultrafilters of $\NN$ \cite[Corollary 4.14]{FHS2}.

\begin{notation}
We let $\C'$ be the set of functions on $S'$ of the form $f'_\M$, subscripted if we want to identify the class of models, e.g. $\C'_{\T}$.  An element of $\C'$ is an equivalence class of models under the relation of mutual finite representability, which is generally much weaker than elementary equivalence.  Not only is $\C'$ compact with the topology of pointwise convergence as functions on $S'$, but the partial order encodes finite representability of models.
\end{notation}

Here is a useful observation.

\begin{lemma} \label{T:upward}
$\C'$ contains a greatest (resp. least) element if and only if it is upward-directed (resp. downward-directed).
\end{lemma}

\begin{proof}
Assuming $\C'$ is upward-directed, one can make it into a net, indexing elements by themselves, and apply compactness.  The other claims are similar or obvious.
\end{proof}

\subsection{Local universality} One says that an object is \textit{universal} for a class if all objects in the class embed into it.  We will work with a weaker notion.

\begin{definition}
Working in some axiomatizable class, we say that a model $\M$ is \textit{locally universal} if every model is finitely representable in it; equivalently, $f'_\M$ dominates all other functions in $\C'$.
\end{definition}

Thus Lemma \ref{T:upward} gives a criterion for the existence of locally universal objects.  CEP asks whether $\R$ is locally universal.

\begin{example} \label{T:ex} ${}$
\begin{enumerate}
\item The poset $\C'_\text{inf-dim Banach}$ is upward-directed -- for example, one can produce an upper bound for a pair via any kind of direct sum or tensor product.  It is classical that there are universal objects for separable Banach spaces, for instance $C[0,1]$ and $C(\{0,1\}^\NN)$, and any of these will \textit{a fortiori} be locally universal.  But there are also locally universal objects that are not universal (for separable spaces), such as $c_0$.

    By Dvoretzky's theorem $f'_{\ell^2}$ is the least element of $\C'_\text{inf-dim Banach}$.
\item Similarly $\C'_{\T}$ is upward-directed -- constructions of an upper bound include the tensor product and (tracial) free product.  Thus \textit{there is a locally universal $\T$ factor}, and it may be taken separable by Theorem \ref{T:dls}.  This contrasts with Ozawa's theorem that there is no universal object for separable $\T$ factors \cite{O1}.

    Since $\R \hookrightarrow \M$ for every $\M$, $f'_\R$ is the least element in $\C'_{\T}$.  CEP asks whether it is also the greatest, i.e., whether $\C'_{\T}$ is a singleton (equivalently, whether all $\T$ factors are locally universal).

    Ozawa proved his theorem by appealing to an uncountable family of group von Neumann factors associated to quotients of a certain property $(T)$ group.  A different argument was given in \cite{NPS}, this time using an uncountable family of crossed product factors.  (The latter factors were additionally shown to be embeddable in $\R^\U$, which is why \cite{NPS} is used in our proof of Theorem \ref{T:main}.)  We mention here that Ozawa's theorem also follows directly from an abstract result in model theory that does not require exhibition of a specific uncountable family.  This will be presented in a sequel paper.
\end{enumerate}
\end{example}

Example \ref{T:ex}(2) may be viewed as a poor man's resolution of CEP, since a locally universal $\T$ factor $\s$ has the property that every $\T$ factor embeds in an ultrapower of $\s$.  Equivalently, any separable $\T$ factor embeds in $\s^\U$, for any $\U \in \beta \NN \setminus \NN$.

Note that locally universal $\T$ factors are not all elementarily equivalent.  Given any separable locally universal $\s$, both $\s \otimes \R$ and $\s * \R$ contain $\s$ so are also locally universal; the first is McDuff and the second does not even have $\Gamma$ \cite[Theorem 11]{Ba}.  Still, if one could somehow identify a specific locally universal $\T$ factor, CEP would become the concrete problem of deciding whether it is finitely representable in $\R$.  The construction implied by Example \ref{T:ex}(2) and Lemma \ref{T:upward} is as follows: form an ultraproduct of representatives of $\C'_{\T}$, where the ultrafilter is a limit point of the net $\C'_{\T}$ itself, and use downward L\"{o}wenheim-Skolem if a separable object is desired.  One can replace the ultraproduct with a tensor product or free product, and then actually countably many factors suffice, so the output is separable with no need for downward L\"{o}wenheim-Skolem.  Just take a countable dense set of $\sup$-sentences $\varphi$ in $S'$, and for each $n$ and $\varphi$ find a $\T$ factor where $\varphi$ is within $\frac1n$ of the supremum of $\varphi$ over all $\T$ factors.

\smallskip

It can also be interesting to include the finite-dimensional objects and consider the structure of $\C'$.  One can ask whether the set $\{f_\M | \: \M \text{ finite-dimensional}\}$ is cofinal in $\C'$, i.e., whether any model embeds in an ultraproduct of finite-dimensional models.  For a Banach space $\X$ one can use the finite-dimensional subspaces of $\X$, but the analogous argument does not work for finite factors because the finite-dimensional subfactors of a $\T$ factor are not upward-directed.

\begin{remark} \label{T:micro}
Here we use the framework of $\C'_{\T}$ to reobtain a few known facts.
\begin{enumerate}
\item CEP is equivalent to asking whether any $\T$ factor embeds into an ultraproduct of matrix algebras, because any infinite-dimensional matricial ultraproduct $\Pi_\U \Mn$ is mutually finitely representable with $\R$.  To see this, use the fact that $\Mn$ and $\R$ both embed in any $\T$ factor, and compute in $\C'_{\text{finite factors}}$:
\begin{equation} \label{E:moment}
f'_\R \leq f'_{\Pi_\U \Mn} = \lim f'_{\Mn} \leq f'_\R \qquad \Rightarrow \qquad f'_\R = f'_{\Pi_\U \Mn},
\end{equation}
cf. \cite[Theorem 3.3]{Li}.
\item It was asked at the end of Section 4 of \cite{FH} whether $\R^\U$ embeds into a matricial ultraproduct.  This is known, and of course it follows from the preceding item, but one can also give a direct argument that avoids model theory.  Write $\R$ as the closure of an increasing union of $\MM_{2^n}$, and let $E_{\MM_{2^n}}: \R \to \MM_{2^n}$ be the trace-preserving conditional expectation.  Then embed $\R$ in a matricial ultraproduct via $\R \ni x \mapsto (E_{\MM_{2^n}}(x)) \in \Pi_\U \MM_{2^n}$, and take an ultrapower of this embedding.
\item Combining \eqref{E:embed} and \eqref{E:moment} with the minimality of $f'_\R$ in $\C'_{\T}$ gives, for any $\T$ factor $\M$,
\begin{equation} \label{E:micro}
\M \hookrightarrow \R^\U \qquad \iff \qquad f'_\M = f'_{\Pi_\U \Mn}.
\end{equation}
Now quantifier-free formulas for $\T$ factors are functions of \textit{moments} (traces of polynomials in variables), so the value of a $\sup$-sentence on $\M$ is a restriction on the set of moments that a tuple from $\M$ can have.  From \eqref{E:micro}, $\M \hookrightarrow \R^\U$ says that any finite set of moments of a tuple from $\M$ can be arbitrarily well-approximated by the moments of a tuple from some matrix algebra.  So CEP is equivalent to asking whether ``moment-wise approximation by matrices" is possible for any tuple from any $\T$ factor, a statement known as the \textit{Microstates Conjecture} \cite[7.4]{V}.
\end{enumerate}
\end{remark}

\begin{remark}
A property is called a \textit{superproperty} if it passes from a model $\M$ to any model finitely representable in $\M$.  Superproperties are in 1-1 correspondence with the \textit{hereditary} subsets of $\C'$, and as such they form a complete lattice.  (For any superproperty $P$, the set $\{f'_\M | \: \M \text{ has $P$}\}$ is a hereditary subset of $\C'$.  On the other hand, for any hereditary $H \subseteq \C'$, the class $\{\M \mid \: f'_\M \in H\}$ defines a superproperty.)  Superproperties are always local but not necessarily axiomatizable.  For example, uniform non-squareness corresponds to the open hereditary subset $\{f'_\X \mid \: f'_\X(-\psi) < 0\}$ of $\C'_{\text{Banach}}$, with $\psi$ as in Section 2.3 (note $-\psi$ is a $\sup$-sentence taking nonpositive values).  CEP asks whether there are no nontrivial superproperties for $\T$ factors; if there are any, finite representability in $\R$ is the strongest one, and the negation of local universality is the weakest one (see \cite[Lecture 16]{Sch} for the Banach space version).  If CEP turns out to be negative, one could move on to study the lattice of superproperties, and see which are axiomatizable.  Note that finite representability in $\R$ is axiomatizable; can one find an explicit axiomatization?  As in Remark \ref{T:micro}(3), this amounts to describing the set of matricial moments (see, e.g., \cite{R1999}).  Recently Netzer and Thom \cite{NT} used old polynomial identities of Amitsur-Levitzki \cite{AL} to prove that a weaker kind of finite representability is universal.
\end{remark}

\appendix

\section{Proof of Proposition \ref{T:gamma1}}

In addition to its use in the present paper, Proposition \ref{T:gamma1} proves, and improves, some facts mentioned in \cite[Example 4.2]{S2009}.  We will employ the following two \textit{ad hoc} quantities $a_\M$, $b_\M$ for a tracial von Neumann algebra $(\M, \tau)$.  Let $a_\M$ be zero if there are no minimal projections, and otherwise the largest trace of a minimal projection.  Let $b_\M = \inf \{|\tau(u)| : \: u \in \U(\M)\}$.

\begin{lemma} \label{T:gamma}
Let $(\M, \tau)$ be a tracial von Neumann algebra.
\begin{enumerate}
\item If $p\in \M$ is a minimal projection with $\tau(p)>$ 1/2, then $p$ is central.
\item The infimum in the definition of $b_\M$ is achieved, and $b_\M  = (2a_\M-1) \vee 0$.
\end{enumerate}
\end{lemma}

\begin{proof}
(1): Let $z(p)$ be the central support of $p$ (the smallest central projection dominating $p$), and suppose toward a contradiction that $z(p) \neq p$.  Then $z(p) p^\perp \neq 0$, so that $p$ and $p^\perp$ are not centrally orthogonal.  By \cite[Lemma 1.7]{T}, $p$ and $p^\perp$ have nonzero equivalent subprojections.  But minimality means that the only nonzero subprojection under $p$ is $p$ itself, which cannot be equivalent to a subprojection of $p^\perp$ because $\tau(p) > \tau(p^\perp)$.  We conclude $z(p) = p$, so that $p$ is central.


(2): If $a_\M \leq 1/2$, then there are three projections $\{p_j\}_{j=1}^3$ adding to $I$ each with trace $\leq 1/2$.  There is a (possibly degenerate) triangle in the complex plane with these traces as leg lengths, so there are complex unit scalars $\{\alpha_j\}_{j=1}^3$ with $\sum \alpha_j \tau(p_j) = 0$.  This makes $\sum \alpha_j p_j$ a trace zero unitary.

Now suppose there is a minimal projection $p$ with $\tau(p) > 1/2$.  Given any unitary $u \in \M$, $pup = \lambda p$ for some unit scalar $\lambda \in \CC$.  This implies
$$|\tau(u)| = |\tau(u(p + p^\perp))| = |\tau(pup) + \tau(u p^\perp)| \geq |\tau(\lambda p)| - |\tau(u p^\perp)| \geq \tau(p) - \tau(p^\perp) = 2 \tau(p) - 1.$$
The bound is achieved because the unitary $p - p^\perp$ has trace exactly $2 \tau(p) - 1$.
\end{proof}



\begin{proof}[Proof of Proposition \ref{T:gamma1}]
To establish the equivalence of the eight conditions, we only need to show that the weakest implies the strongest.  So assume that the relative commutant in $\M^\U$ of any finite set in $\M$ is nontrivial, and seeking a contradiction, let $\{(x^n_j)_j\}_{n=1}^\infty \subset \M^\U$ be such that $W^*(\{(x^n_j)_j\})' \cap \M^\U$ does not contain a trace zero unitary.  The goal of the next two paragraphs is to find a finite set in $\M$ that generates an irreducible subfactor of $\M$, so that a result from \cite{FGL} can be used to deduce a contradiction.  We find the set by first taking finitely many elements directly out of the representing sequences $\{(x^n_j)_j\}$, then showing that their relative commutant in $\M^\U$ contains a large central projection, and finally adding an appropriate unitary.

Now for any $j$, by hypothesis $\N_j \overset{\text{def}}{=} W^*(\{x^n_j\}_{n=1}^j)' \cap \M^\U$ is nontrivial.  Find a unitary in $\N_j$ whose trace is $< \frac1{j} + b_{\N_j}$.  By going far enough out in the representing sequence for this unitary, we can find $u_j \in \U(\M)$ that nearly commutes with $\{(x^n_j)_j\}$ and has trace $< \frac2{j} + b_{\N_j}$.  Then $(u_j)$ is in the relative commutant $W^*(\{(x^n_j)_j\})' \cap \M^\U$, with trace $\lim_{j \to \U} b_{\N_j}$, and by assumption this trace cannot be zero.  Thus, for $j$ in a neighborhood of $\U$, the value of $b_{\N_j}$ is bounded away from zero.  From here on we only discuss those $j$ in this neighborhood.  By Lemma \ref{T:gamma}(2), $\N_j$ contains a minimal projection of trace $> \alpha > 1/2$.

In this paragraph we are able to work inside $\M$.  The relative commutant $\M_j \overset{\text{def}}{=} W^*(\{x^n_j\}_{n=1}^j)' \cap \M$ is contained in $\N_j$, so it also contains a minimal projection $p_j$ of trace $> \alpha$.  By Lemma \ref{T:gamma}(1), $p_j$ is central in $\M_j$.  Let $u_j$ be a unitary in $\M$ such that $p_j^\perp \leq u_j p_j u_j^* $.  The relative commutant of the finite set $\{x^n_j\}_{n=1}^j  \cup \{u_j\}$ is $\M_j \cap u_j \M_j u_j^*$, and we claim that it is trivial.  For suppose $y \in \M_j \cap u_j \M_j u_j^*$.  From $y \in \M_j$ and the centrality and minimality of $p_j$, we know that $p_j y = \lambda p_j$ for some $\lambda \in \CC$.  Applying the same argument to the condition $y \in u_j \M_j u_j^*$, we deduce $(u_j p_j u_j^*) y = \mu u_j p_j u_j^*$ for some $\mu \in \CC$.  Now
\begin{align*}
&y = (p_j + p_j^\perp) y = \lambda p_j + p_j^\perp (u_jp_ju_j^*) y = \lambda p_j + p_j^\perp \mu u_jp_ju_j^* = \lambda p_j + \mu p_j^\perp;
\\&y = (u_jp_ju_j^* + u_jp_j^\perp u_j^*) y = \mu u_jp_ju_j^* + u_jp_j^\perp u_j^* (p_j) y = \mu u_j p_ju_j^* + u_jp_j^\perp u_j^* \lambda p_j = \mu u_jp_ju_j^* + \lambda u_jp_j^\perp u_j^*.
\end{align*}
Taking the trace of these two equalities gives
$$\lambda \tau(p_j) + \mu(1 - \tau(p_j)) = \tau(\lambda p_j + \mu p_j^\perp) = \tau(y) = \tau(\mu (u_jp_ju_j^*) + \lambda (u_jp_ju_j^*)^\perp) = \mu \tau(p_j) + \lambda(1 - \tau(p_j)),$$
and since $\tau(p_j) \neq 1/2$, this forces $\lambda = \mu$ and finally $y = \lambda I$.

Thus $\{x^n_j\}_{n=1}^j  \cup \{u_j\}$ generates an irreducible subfactor in $\M$.  By \cite[Lemma 3.5]{FGL}, its relative commutant in $\M^\U$ is either trivial or nonatomic.  (Although the introduction to \cite{FGL} restricts consideration to factors with separable predual, this is not required for Lemma 3.5 or its supporting material.)  By assumption this relative commutant is not trivial.  We finish the proof by pointing out that it also cannot be nonatomic.  In fact it is just $\N_j \cap u_j \N_j u_j^*$, which is contained in $\N_j$, which has a nontrivial minimal projection; it follows that the relative commutant also has a nontrivial minimal projection.
\end{proof}

\end{document}